\documentclass[10pt]{amsart} 
\usepackage{amsmath,amssymb, amsthm,amsbsy, amsfonts,amsthm,amscd, amssymb,amscd,mathrsfs} 
\usepackage[dvips]{graphicx}
\usepackage[usenames,dvipsnames]{pstricks} 
\usepackage{epsfig}
\usepackage[subnum]{cases}
\usepackage[colorlinks=true, linkcolor=magenta, citecolor=cyan, urlcolor=blue,hyperfootnotes=false]{hyperref}
\usepackage{pst-grad} % For gradients 
\usepackage{pst-plot} % For axes
\newcommand{\R}{{\mathbb R}}
\newcommand{\N}{{\mathbb N}} 
\newcommand{\T}{{\mathbb T}} 

\newcommand{\C}{{\mathbb C}}

\newcommand{\Lin}{\mathcal{L}}
\newcommand{\Hin}{\mathcal{H}}
\newcommand{\Win}{\mathcal{W}}

 \renewcommand{\geq }{\geqslant}
 \renewcommand{\leq }{\leqslant}
\usepackage{mathtools}
\DeclarePairedDelimiter{\abs}{\lvert}{\rvert}

\DeclarePairedDelimiter{\norma}{\lVert}{\rVert} 
\usepackage{braket}
\usepackage{color} 
\newcommand{\Rn}{{\mathbb R^n}}
\newenvironment{sistema}%
{\left\lbrace\begin{array}{@{}l@{}}}% 
{\end{array}\right.}  
\numberwithin{equation}{section}
\newtheorem{theorem}{Theorem}[section]
\newtheorem{corollary}[theorem]{Corollary}
\newtheorem{lemma}[theorem]{Lemma}
\newtheorem{proposition}[theorem]{Proposition}
\theoremstyle{definition} 
\newtheorem{definition}[theorem]{Definition}

\newtheorem{remark}[theorem]{Remark}

\begin{document} 

\title[$H^1$-scattering for NLS-systems in low dimensions]{$H^1$-scattering for systems of $N$-defocusing weakly coupled NLS equations in low space dimensions}

%%%%%%%%%%%%%%%%%%%%%%%%AUTHORS%%%%%%%%%%%%%%%%%%%%%%%%

\author{B.~Cassano}
\address{Biagio Cassano: SAPIENZA Universit$\grave{\text{a}}$ di Roma, Dipartimento di Matematica, P.le A. Moro 5, 00185-Roma, Italy}
\email{cassano@mat.uniroma1.it}

\author{M.~Tarulli}
\address{Mirko Tarulli: Universit$\grave{\text{a}}$ Degli Studi di Pisa, Dipartimento di Matematica,
Largo Bruno Pontecorvo 5 I - 56127 Pisa. Italy}
\email{tarulli@mail.dm.unipi.it}


\thanks{
The authors were supported by the Italian project FIRB 2012 {\it Dispersive Dynamics: Fourier Analysis and Variational Methods.}
}
%%%%%%%%%%%%%%%%%%%%%%%%%%%%END OF AUTHORS%%%%%%%%%%%%%

\begin{abstract}
We prove that the scattering operators and wave operators  are well-defined in the energy space for the system of defocusing Schr\"odinger equations 
\begin{equation*}
\begin{cases}
i\partial_t u_\mu + \Delta u_\mu -
\displaystyle{\sum_{\substack{\mu,\nu=1 }}^N}
\beta_{\mu\nu}|u_\nu|^{p+1}|u_\mu|^{p-1}u_\mu=0, 
\quad\quad \mu=1,\dots,N, \\
(u_\mu(0,\cdot))_{\mu=1}^N= (u_{\mu,0})_{\mu=1}^N \in H^1(\R^d)^N.
\end{cases}
\end{equation*}
with $N\geq 2$, $\beta_{\mu\nu} \geq 0$, $\beta_{\mu\mu}\neq 0$
for $p>2 $ if $d=1$, $p>1 $ if $d=2$ and $ 1 \leq p < 2$ if $d=3$.
\end{abstract}

\date{\today}
\maketitle

\section{Introduction}\label{sec:introduction}

The main object of the paper is the study of the decaying and scattering properties of the solution to the following system of $N\geq 2$ defocusing nonlinear Schr\"odinger equations in dimension $1\leq d \leq 3 $:
\begin{equation}\label{eq:nls}
\begin{cases}
i\partial_t u_\mu + \Delta u_\mu -
\displaystyle{\sum_{\substack{\mu,\nu=1 }}^N}
\beta_{\mu\nu}|u_\nu|^{p+1}|u_\mu|^{p-1}u_\mu=0, 
\quad\quad \mu=1,\dots,N, \\
(u_\mu(0,\cdot))_{\mu=1}^N= (u_{\mu,0})_{\mu=1}^N \in H^1(\R^d)^N.
%,v(0,x))=(u_0(x),v_0(x))\in H^1(\R^d)\times H^1(\R^d). \nonumber
\end{cases}
\end{equation}
Here, for all $\mu,\nu=1,\dots,N$, $u_\mu=u_\mu(t,x):\R\times\R^d\to\C$, $(u_\mu)_{\mu=1}^N=(u_1,\dots, u_N)$ and  $\beta_{\mu\nu}\geq0$, $\beta_{\mu\mu}\neq 0$
 are coupling parameters, moreover we require that the nonlinearity parameter $p$ satisfies the following condition
\begin{align}\label{eq:base}
&1 \leq p<p^*(d), \ \ \ \  p^*(d)=
\begin{cases}
+\infty \ \ \ \ \ \, & \text{if} \ \ \ d=1,2, \\
2  \ \ \ \  \  \  &\text{if} \ \ \ d=3.
\end{cases},  \\
&  \frac{2}{d}< p. \label{eq:al}
\end{align}
Take into account that the power nonlinearity $p^*(d)$ corresponds to the  $H^1$-critical exponent for the single NLS in $\R^{d}$, while the lower bound $\max(1, \frac 2 d)$ arises from technical limitations associated to the well-posedness in the product space $H^1(\R^d)^N$ for the solutions to \eqref{eq:nls}, as we see later in the Remarks \ref{beta0} and \ref{decaylim}. 
There is a vast literature regarding the global well-posedness theory as well as the 
bound state theory for the problem \eqref{eq:nls}, on the other hand
the system of Schr\"odinger equations   plays an important role in many models of mathematical physics: 
it describes the interactions of $M$--wave packets, the nonlinear waveguides, 
the optical pulse propagation in birefringent fibers, 
the propagation of polarized laser beam in Kerr-like photorefractive media and in the Bose-Einstein condensates theory, 
just to name a few. We remand to \cite{Col}, \cite{FaMo}, \cite{Pomponio}, \cite{MMP} and \cite{FaLuMo}
in the case $N=2$ and to \cite{LinWei} and \cite{NTDS} in the general case $N\geq 2$
for a complete set of references both on mathematical and on physical setting and applications. We analyze here the scattering in $H^1(\R^d)^N$ for \eqref{eq:nls}, in analogy with the case of the single defocusing Schr\"odinger equation
\begin{equation}\label{eq:nle}
  \begin{sistema}
    i \partial_t u + \Delta u - \abs{u}^{2p}u=0\\
    u(0)=u_0 \in H^1(\R^d),
  \end{sistema}
\end{equation}
with $u : \R \times \R^d \to \C$ and $p>0$, following the spirit of the paper \cite{Vis}. The basic new idea in this paper, among other results, is to use the interaction Morawetz etimates for exploiting the decay of $L^q$-norms of the solutions to \eqref{eq:nle} as $t\rightarrow \pm \infty$, provided $2<q<6$ for $d=3$ and $2<q<\infty$ for $d=1,2$. Our main aim is to present an analogous behavior for \eqref{eq:nls}, 
introducing some novelties. Namely, in a first step we perform new Morawetz identities, interaction Morawetz identities (which appear, at least in the seminal form and for a single NLS in the paper \cite{PlVe}) and their corresponding inequalities for the system \eqref{eq:nls} easing the proofs given in \cite{Vis}. Then by localizing the nonlinear part of Morawetz inequalities above on space-time cubes we are in position, as a second step, to give a contradiction argument which enable us to say that the solutions $(u_{\mu})_{\mu=1}^N$ behave exactly as in Theorem 0.1 in \cite{Vis}. Such as phenomenon, in combination with a generalization of the nonlinear theory developed in \cite{Ca} to the system on coupled NLS, brings to have asymptotic completeness and existence of the full wave operators in the energy space $H^1(\R^d)^N$ for solution to \eqref{eq:nls}. We emphasize that our result rely on an argument which yields the asymptotics in a single stroke and which does not distinguish the number $N$ of coupled equations. In fact, by  writing the linear part of the interaction Morawetz in a appropriate form and dealing only with its nonlinear part, it is possible to overcome the mathematical difficulties, and moreover to provide a further simple proof of scattering results appearing in \cite{GiVel2},  \cite{PlVe} and specially in \cite{Nakanishi1}. In this last paper, the author produces a
set of weighted Morawetz estimate and uses the \emph{separation of localized energy method} to achieve that the wave operators and the scattering 
operators for \eqref{eq:nle} when $d=1,2$ are well-defined and bijective in $H^1$, but this is very difficult to extend to a system of coupled NLS. 
\\
 %Morawetz:
%\cite{Ca}, \cite{Morawetz}, Lin Strauss \cite{LinStraus}, \cite{GiVel2},
%\cite{BarceloRuizVega} \cite{PerthameVega}
% (con termine magnetico \cite{FanelliVega}, \cite{Goncalves}, \cite{Garcia})
% Interaction (bilinear) Morawetz:
%\cite{TaoVisanZhang}, \cite{PlVe}, \cite{CGT}, the survey \cite{GinVel}

We are in position now to enter in the heart of matter by stating the main result of this paper, that is 
%\begin{equation}\label{eq:al}
%\frac{2}{n}\leq p<p^*(d), \ \ \ \  p^*(d)=
%\begin{cases}
%\frac{2}{n-2}  \ \ \ \  \  \  \text{if} \ \ \ d \geq 3,\\
%+\infty \ \ \ \ \ \,  \text{if} \ \ \ d=1,2.
%\end{cases}
%\end{equation}
\begin{theorem}\label{thm:main}
  Let $1\leq d \leq 3$, $p\in\R$ such that \eqref{eq:base}, \eqref{eq:al} hold,
  then:
  \begin{itemize}
  \item\emph{(asymptotic completeness)} If $(u_{\mu,0})_{\mu=1}^N \in H^1(\R^d)^N,$ then the unique global solution to \eqref{eq:nls} 
    $(u_{\mu})_{\mu=1}^N \in\mathcal C(\R, H^1(\R^d)^N)$ 
    scatters, i.e. there exist $(u_{\mu,0}^{\pm})_{\mu=1}^N \in H^1(\R^d))^N$ such that for all $\mu=1,\dots,N$ 
    \begin{equation}\label{eq:scattering}
      \lim_{t\to\pm\infty}\left\|u_{\mu}(t,\cdot)-e^{it\Delta}u_{\mu,0}^\pm(\cdot)\right\|_{H^1}=0
      \quad
      \text{for all $\mu=1,\dots,N$}.
    \end{equation} 
  \item\emph{(existence of wave operators)} For every $(u_{\mu,0}^{\pm})_{\mu=1}^N \in H^1(\R^d)^N$ there exist unique 
    initial data $(u_{\mu,0})_{\mu=1}^N \in H^1(\R^d)^N,$ such that the
    global solution to \eqref{eq:nls} 
    $(u_{\mu})_{\mu=1}^N \in\mathcal C(\R, H^1(\R^d)^N)$  satisfies
    \eqref{eq:scattering}.
  \end{itemize}

\end{theorem}
\begin{remark}[Case $\beta_{\mu\nu}=0$, $\mu\neq \nu$]\label{beta0} 
  If, for $\mu\neq\nu$, some of the $\beta_{\mu\nu}$  is nonvanishing,
  in order to treat the coupling nonlinearity considered in \eqref{eq:nls}, we 
  are forced to assume $p\geq1$: this excludes the analisys of the system 
  in dimension $d\geq 4$, since in this case an existence theorem is not available 
  (see Prop.~\ref{ConsLaw}). By the way, in the trivial case $\beta_{\mu\nu}=0$ 
  for all $\mu\neq\nu$, we are no longer
  obliged to assume $p\geq 1$, and hence, as a byproduct of this paper, we get 
  decay (for $0<p<2/(d-2)$) and scattering (for $2/d<p<2/(d-2)$) results for the solution to the Cauchy problem \eqref{eq:nle} in all dimensions $d \geq 1$.
  We remark that such results were already 
  established in \cite{Vis}, however our technicalities simplifies 
  some arguments present in it. Finally, we underline that this approach eases the well known results 
   \cite{Nakanishi1,Nakanishi2} for the scattering of \eqref{eq:nle} 
  in lower dimension $d=1, 2$.
\end{remark}

As far as concerns the literature, to our knowledge Morawetz and interaction Morawetz estimates
are not available in the system framework, then we conclude succinctly by recalling some of the known results, other than the already cited \cite{Vis}, \cite{Nakanishi1} and \cite{Nakanishi2} connected with the problem \eqref{eq:nle}.
This problem is classical and we remand to \cite{Ca} (and references therein) for
a complete exposition.
In order to shed light on scattering properties for solutions to \eqref{eq:nle} it is necessary 
to get fundamental tools such as the Morawetz multiplier technique and the resulting estimates.
These were obtained for the first time in \cite{Morawetz} for the Klein-Gordon equation with a general nonlinearity and 
were successively used for proving  the asymptotic completeness in \cite{LinStrauss}  for the cubic NLS in $\R^3$  and in \cite{GiVel2}
for the Schr\"odinger equation in $\R^d$ and with a pure power nonlinearity as in \eqref{eq:nle} for $2/d<p<2/(d-2)$ (that is, $L^2$-supercritical and $H^1$-subcritical). 
Recently, a new approach has simplified the proof of scattering, consisting
in getting bilinear Morawetz inequalities, also named 
interaction of quadratic Morawetz inequalities, specificly 
Morawetz estimates for two solutions (possibly the same solution taken twice) are computed
at once. We quote in this direction the papers \cite{CKSTT1}, \cite{CKSTT2}, where cubic and quintic defocusing NLS in $\R^3$ are considered, the \cite{CGT} in which interaction Morawetz and then asymptotic completeness are proved for the cubic defocusing NLS in $\R^2$, the paper
\cite{PlVe} where the interaction Morawetz estimates which do not involve the bilaplacian of the Morawetz multipliers are given for the $L^2$-supercritical and $H^1$-subcritical NLS in $\R^d$ with $d\geq1$, providing also application to various nonlinear problems also settled on $3D$ exterior domains and finally the survey \cite{GiVel} where the authors show quadratic Morawetz estimates and scattering for the NLS and the Hartree equation in the $L^2$-supercritical and $H^1$-subcritical cases.
We quote also \cite{DFVV}, \cite{FaVe}, \cite{Goncalves}, \cite{Garcia} (and references therein), where such a theory is applied considering the presence of electromagnetic potentials
and the paper \cite{TzVi}, where the interaction Morawetz technique is extended
to the partially periodic setting in the scattering analysis of the NLS posed on the product space $\R^d\times\T,$ with $d\geq1$.

\subsection*{Outline of paper.} In Section \ref{InterMor} we establish the interaction Morawetz identities 
and inequalities
(in Lemmas \ref{lem:intmor})
and the corresponding Morawetz estimates (in Propositions \ref{dim3} and \ref{dim12})
for the system of NLS \ref{eq:nls} ancillary for proving the Theorem \ref{thm:main}. 
The Section \ref{MainThm} is divided in two part: in the former we show how the interactive Morawetz inequalities 
give a relevant advantage in the exploitation of the decay of solutions to \ref{eq:nls}, this is contained in the Proposition \eqref{decay}, which has its own interest; 
in the latter we look at the existence of scattering states and wave 
operators by an extension of scattering techniques to the systems frame. 
Finally in the Appendix \ref{appendix} a generalized Gagliardo-Nirenberg inequality is obtained 
(see for instance \cite{Vis}), which is primary instrument in our paper arising from the classical Fourier theory.

\subsection*{Notations.}
For any $1\leq r\leq\infty$ we denote by $1\leq r' \leq \infty$ its H\"older conjugate exponent. We indicate by $L_x^r$ the Lebesgue space $L^r(\Rn)$, and respectively by $W^{1,r}_x$ and  $H^1_x$ the  inhomogeneous Sobolev spaces $W^{1,r}(\Rn)$ and $H^1(\Rn)$ (for more details see \cite{Ad}).
For any $N\in \N $, we also set $ \Lin_x^r= L^r(\Rn)^N$ and define 
the Sobolev spaces $\Win^{1,r}_x=W^{1,r}(\Rn)^N$ and $\Hin^1_x=H^1(\Rn)^N.$ For any differential operator $\mathscr D$ we utilize the symbol $\mathscr D_x$ (resp. $\mathscr D_y$) to explicit the dependence on the $x$ (resp. $y$) variable.

\section{Morawetz and interaction Morawetz identities}\label{InterMor}
We provide in this section the fundamental tools for the proof of our main theorem. 
We start by obtaining Morawetz-type identities, which are similar to the ones which hold for the single NLS. 
We find convenient to introduce the following notations: given a function $f\in H^1(\R^d,\C)$, we denote by
\begin{equation}\label{eq:notation}
  m_f(x):=|f(x)|^2,
  \qquad
  j_f(x):=\Im\left[\overline{f}\nabla f(x)\right]\in\C^d.
\end{equation}
We have
\begin{lemma}[Morawetz]\label{lem:mor}
Let $d\geq1$, and $(u_\mu)_{\mu=1}^N \in\mathcal C(\R,H^1(\R^d)^N)$ be a global solution to system \eqref{eq:nls},
 let $\phi=\phi(x):\R^d\to\R$ be a sufficiently regular and decaying function, and denote by
\begin{equation*}
  V(t):=\sum_{\mu=1}^N\int_{\R^n}\phi(x) \, m_{u_\mu}(x)\,dx.
\end{equation*}
The following identities hold:
\begin{align}
\dot V(t)
=&\sum_{\mu=1}^N \int_{\R^d}\phi(x) \dot m_{u_\mu}(x)\,dx
=2\sum_{\mu=1}^N \int_{\R^d} j_{u_\mu}(x)  \cdot\nabla\phi(x)\,dx
\label{eq:mor1}
\\
\ddot V(t)
=&\sum_{\mu=1}^N \int_{\R^d}\phi(x)\ddot m_{u_\mu}(x)\,dx
\label{eq:mor2}
\\
=&\sum_{\mu=1}^N\left[-\int_{\R^d} m_{u_\mu}(x)\Delta^2\phi(x)\,dx
+4  \int_{\R^d}\nabla u_{\mu}(x)D^2\phi(x)\cdot\nabla \overline u_{\mu}(x)\,dx \right]
\nonumber
\\
&
+\frac{2p}{p+1}
\sum_{\mu,\nu=1}^N \beta_{\mu\nu}\int_{\R^d}|u_{\mu}(x)|^{p+1}|u_{\nu}(x)|^{p+1}\Delta\phi(x)\,dx ,
\nonumber
\end{align}
where $D^2\phi\in\mathcal M_{n\times n}(\R^d)$ is the hessian matrix of $\phi$, and $\Delta^2\phi=\Delta(\Delta\phi)$ the bi-laplacian operator.
\end{lemma}
\begin{proof}
  We prove the identities for a smooth solution $(u_\mu)_\mu$, letting the general case 
  $(u_\mu)_{\mu=1}^N \in\mathcal C(\R,H^1(\R^d)^N)$ to a final standard density argument 
  (see for instance \cite{Ca}, Theorem 7.6.4, Step 2, or \cite{GiVel}, Appendix 4). 
  The equation \eqref{eq:mor1} is easy to check. We give some details for obtaining \eqref{eq:mor2}.
  By means of an integration by parts and thanks to \eqref{eq:nls}, we have for every fixed $\mu$
  \begin{equation}\label{eq:dausare1}
    \begin{split}
    &2\partial_t \int_{\R^d} j_{u_\mu}(x)\cdot \nabla \phi(x) \,dx  \\
    &\qquad =- 2\Im  \int_{\R^d} \partial_t u_{\mu}(x) [\Delta \phi(x)\bar u_{\mu}(x)+2\nabla \phi(x)\cdot \nabla \bar u_{\mu}(x)]\,dx \\
    &\qquad=2\Re \int_{\R^d} i \partial_t u_{\mu}(x) [\Delta \phi(x)\bar u_{\mu}(x)+2\nabla \phi(x)\cdot \nabla \bar u_\mu(x)]\,dx\\ 
    &\qquad=2\Re \int_{\R^d} \Big[-\Delta u_{\mu}(x) + 
    \sum_{\nu=1}^N \beta_{\mu\nu}\abs{u_\nu(x)}^{p+1}\abs{u_\mu(x)}^{p-1} u_\mu(x)\Big]\\
    &\qquad\qquad\qquad\qquad\cdot[\Delta \phi(x)\bar u_\mu(x)+2\nabla \phi(x)\cdot \nabla \bar u_\mu(x)]\,dx.
    \end{split}
  \end{equation}
  We have
  \begin{equation}\label{eq:dausare2}
    \begin{split}
    2\Re\int_{\R^d} &-\Delta u_{\mu}(x) [\Delta\phi(x)\bar u_{\mu}(x)+2\nabla \phi(x)\cdot \nabla \bar u_{\mu}(x)]\,dx\\
    &= - \int_{\R^d} \Delta^2 \phi(x) \abs{u_{\mu}(x)}^2\,dx + 4\int_{\R^d} \nabla u_{\mu}(x) D^2 \phi(x)\nabla \bar u_{\mu}(x)\,dx.
    \end{split}
  \end{equation}
Moreover
\begin{equation*}
  \begin{split}
    2\sum_{\nu=1 }^N \beta_{\mu\nu}  \Re \int_{\R^d} \abs{u_\nu}^{p+1}\abs{u_\mu}^{p-1} u_\mu(x)
    \cdot[\Delta \phi(x)\bar u_\mu(x)+2\nabla \phi(x)\cdot \nabla \bar u_\mu(x)]\,dx &\\
    = 2\sum_{\nu=1 }^N \beta_{\mu\nu}  \Re  \int_{\R^d} \abs{u_\mu u_\nu}^{p+1}\Delta \phi(x)
    +  2\nabla \phi(x) \cdot \frac{\nabla\abs{u_\mu}^{p+1}}{p+1} \abs{u_\nu}^{p+1}\,dx,&
  \end{split}
\end{equation*}
and, summing over $\nu,\mu=1,\dots,N$,
\begin{equation}
  \label{eq:dausare3}
  \begin{split}
    &2\sum_{\nu,\mu=1}^N \beta_{\mu\nu}  \Re  \int_{\R^d} \abs{u_\mu u_\nu}^{p+1}\Delta \phi(x)
    +  \nabla \phi(x) \cdot \frac{2\nabla\abs{u_\mu}^{p+1}}{p+1}\abs{u_\nu}^{p+1} \,dx  \\     
%    &2\sum_{\mu,\nu} \beta_{\mu\nu}  \Re  \int_{\R^d} \abs{u_\mu u_\nu}^{p+1}\Delta \phi\,dx \\
    =&\, 2\sum_{\nu,\mu=1}^N \beta_{\mu\nu} \Re  \int_{\R^d} \abs{u_\mu u_\nu}^{p+1}\Delta \phi(x) 
    +\nabla \phi(x) \cdot \frac{\nabla(\abs{u_\mu}^{p+1}  \abs{u_\nu}^{p+1})}{p+1} \,dx  \\
    =&\, 2\sum_{\nu,\mu=1}^N \beta_{\mu\nu}\left(1 - \frac{1}{p+1}\right) 
      \Re  \int_{\R^d} \abs{u_\mu u_\nu}^{p+1}\Delta \phi(x)  \,dx,
  \end{split}
\end{equation}
where in the last equality we have used integration by parts.
Taking in account \eqref{eq:dausare1}, \eqref{eq:dausare2}, summing
over $\mu=1,\dots,N$, and considering \eqref{eq:dausare3}, we get the thesis.
\end{proof}

By means of the previous Lemma, we can now prove the following interaction Morawetz identities.

\begin{lemma}[Interaction Morawetz]\label{lem:intmor}
  Let $(u_\mu)_{\mu=1}^N \in\mathcal C(\R, H^1(\R^d)^N)$ be a global solution to system \eqref{eq:nls}, let $\phi=\phi(|x|):\R^d\to\R$ be a convex radial function, regular and decaying enough, and denote by $\psi=\psi(x,y):=\phi(|x-y|):\R^{2d}\to\R$,
\begin{equation*}
  I(t):=\sum_{\mu,\kappa=1}^N\int_{\R^d}\int_{\R^d}\psi(x,y)m_{u_\mu}(x)m_{u_\kappa}(y)\,dx\,dy.
\end{equation*}
The following holds:
\begin{align}
  \dot I(t)
  &
  =2\sum_{\mu,\kappa=1}^N\int_{\R^d}\int_{\R^d}j_{u_\mu}(x)\cdot\nabla_x\psi(x,y) \, m_{u_\kappa}(y)\,dxdy,
  \label{eq:intmor1}
  \\
  \ddot I(t)
  &
  \geq
  2\sum_{\mu,\kappa=1}^N\int_{\R^d}\int_{\R^d}
  \Delta_x\psi(x,y) \nabla_xm_{u_\mu}(t,x)\cdot\nabla_y m_{u_\kappa}(t,y)\,dxdy+N_{(p,\psi)},
  \label{eq:intmor2}
\end{align}
with
\begin{equation}
\begin{split}
N_{(p,\psi)}=
&
\frac {4p}{p+1}\sum_{\substack{\mu,\nu,\kappa=1}}^N\beta_{\mu\nu}\int_{\R^d}\int_{\R^d}
  |u_\mu(x)|^{p+1}|u_\nu(x)|^{p+1}m_{u_\kappa}(y)\Delta_x\psi(x,y)\,dxdy.
  \label{eq:nonlin}
\end{split}
\end{equation}
\end{lemma}
\begin{proof}
 As for the previous lemma, we prove the identities for a smooth solution $(u_\nu)_{\nu=1}^N$, 
 letting the general case $(u_\mu)_{\mu=1}^N \in\mathcal C(\R,H^1(\R^d)^N)$ to a final standard density argument.
  First one has 
  \begin{align}\label{eq:mor0}
    \dot I(t) =& \sum_{\mu,\kappa=1}^N\int_{\R^d}\int_{\R^d}
 \left( \dot m_{u_\mu}(x)m_{u_\kappa}(y)+ m_{u_\mu}(x) \dot m_{u_\kappa}(y)\right)\psi(x,y)\,dx\,dy,
 %+\\
  %&\int_{\R^d} \int_{\R^d} \left( m_u(x)+m_v(x)\right)\left(\dot m_u(y)+\dot m_v(y)\right)\psi\,dx\,dy,
  %\nonumber
  \end{align}
  then, due to the symmetry of $\psi(x,y)=\phi(|x-y|)$,
  we obtain that the equality above is equivalent to 
  \begin{align*}
    \dot I(t) =2 \sum_{\mu,\kappa=1}^N \int_{\R^d}\int_{\R^d}
  \dot m_{u_\mu}(x)m_{u_\kappa}(y)\psi(x,y)\,dx\,dy.
  \end{align*}
  Therefore, \eqref{eq:intmor1} immediately follows by \eqref{eq:mor1} and the Fubini's Theorem.
  Analogously, we can differentiate again and get the identity
  \begin{align}
  \ddot I(t)
  =
  \sum_{\mu,\kappa=1}^N
  \int_{\R^d}\int_{\R^d}
  \ddot m_{u_\mu}(x)m_{u_\kappa}(y)\psi(x,y)\,dx\,dy
   \nonumber&
    \\
   +\sum_{\mu,\kappa=1}^N\int_{\R^d}\int_{\R^d}
  m_{u_\mu}(x) \ddot m_{u_\kappa}(y)\psi(x,y)\,dx\,dy&
    \label{eq:1}
    \\
    +2\sum_{\mu,\kappa=1}^N\int_{\R^d}\int_{\R^d}
 \dot m_{u_\mu}(x) \dot m_{u_\kappa}(y)\psi(x,y)\,dx\,dy.&
    \nonumber
  \end{align}
  We can write $\ddot I(t):=A+B$: by \eqref{eq:mor2}, an application of the Fubini's Theorem and using once again the
  symmetry of $\psi(x,y)$ we are allowed to set
  \begin{align}
  A =
  -2\sum_{\mu,\kappa=1}^N\int_{\R^d}\int_{\R^d}
  m_{u_\mu}(x)m_{u_\kappa}(y)\Delta^2_x\psi(x,y)\,dxdy&
   \label{eq:2iniz}
  \\
  +\frac{4p}{p+1}\sum_{\mu,\kappa=1}^N\beta_{\mu\mu}\int_{\R^d}\int_{\R^d}
  |u_\mu(x)|^{2p+2}m_{u_\kappa}(y)\Delta_x\psi(x,y)\,dxdy&
  \nonumber
  \\
   \ \ \ +\frac{4p}{p+1}\sum_{\substack{\mu,\nu,\kappa=1\\
  \mu\neq\nu}}^N\beta_{\mu\nu}\int_{\R^d}\int_{\R^d}
  |u_\mu(x)|^{p+1}|u_\nu(x)|^{p+1}m_{u_\kappa}(y)\Delta_x\psi(x,y)\,dxdy,&
  \nonumber
  \end{align}
notice that the second and third line of the \eqref{eq:2iniz} above are sum of terms coming from
the nonlinearity in the equation, while the
r.h.s. of the first line consists of sums of terms related to the linear part of the equation. 
We reshape the linear term in the previous identity \eqref{eq:2iniz} as follows
  \begin{equation}\label{eq:equivalence}
    \begin{split}
    &-2\sum_{\mu,\kappa=1}^N\int_{\R^d}\int_{\R^d} m_{u_\mu}(t,x)m_{u_\kappa}(t,y)\Delta^2\psi(x,y)\,dxdy\\
    =&2\sum_{\mu,\kappa=1}^N\int_{\R^d}\int_{\R^d} m_{u_\mu}(t,x)m_{u_\kappa}(t,y)\partial_{x_i}\partial_{y_i}\Delta\psi(x,y)\,dxdy\\
    =&2\sum_{\mu,\kappa=1}^N\int_{\R^d}\int_{\R^d} \partial_{x_i} m_{u_\mu}(t,x)\partial_{y_i}m_{u_\kappa}(t,y)\Delta\psi(x,y)\,dxdy,
    \end{split}
  \end{equation}
applying integration by parts (with no boundary terms)
and using the property $\partial_{x_k}\psi=-\partial_{y_k}\psi.$
In conclusion, we get
\begin{equation}\label{eq:2}
\begin{split}
A = 2\sum_{\mu,\kappa=1}^N\int_{\R^d}\int_{\R^d}
  \Delta_x\psi(x,y) \nabla_xm_{u_\mu}(t,x)\cdot\nabla_y m_{u_\kappa}(t,y)\,dxdy+N_{(p,\psi)}.
\end{split}
\end{equation}
Moreover by \eqref{eq:mor1}, \eqref{eq:mor2} and the Fubini's Theorem we  introduce
  \begin{align}
  B =
  &
4\sum_{\mu,\kappa=1}^N\int_{\R^d}\int_{\R^d}
  \nabla u_\mu(x)D^2_x\psi(x,y)\nabla\overline u_\mu(x)
  m_{u_\kappa}(y)\,dxdy
  \nonumber
   \\
  &
   +4\sum_{\mu,\kappa=1}^N\int_{\R^d}\int_{\R^d}
  m_{u_\mu}(x)\nabla u_\kappa(y)D^2_y\psi(x,y)\nabla\overline u_\kappa(y)
  \,dxdy
  \nonumber
  \\
  &
  +8\sum_{\mu,\kappa=1}^N\int_{\R^d}\int_{\R^d}
  j_{u_\mu}(x)D^2_{xy}\psi(x,y)\cdot j_{u_\kappa}(y)\,dxdy,
  \nonumber
  \end{align}
here we used, at least at this level, the symmetry of $D^2 \psi$ to eliminate the real part condition in the first two 
summands of the equality above.  Let us focalize on $B:$ it is the sum of two terms, $B_{\mu=\kappa}, $ and $B_{\mu\neq\kappa}$. 
We deal with each of them separately, then we start with the summand with $\mu=\kappa$ that is
 
 \begin{align}\label{eq:b1a}
 B_{\mu=\kappa}=\sum_{\mu=1}^N B^{\mu\mu},&
\end{align}
where, for each $\mu=1,...,N$ the $B^{\mu\mu}$ term is defined by the chain of equalities
\begin{align}\label{eqb1aI}
B^{\mu\mu} = & \,4\int_{\R^d}\int_{\R^d}
  m_{u_{\mu}}(x)\nabla_yu_{\mu}(y)D^2_y\psi(x,y)\nabla_y\overline u_{\mu}(y)\,dxdy \\
  &+ 4\int_{\R^d}\int_{\R^d}
 m_{u_{\mu}}(y) \nabla_xu_{\mu}(x)D^2_x\psi(x,y)\nabla_x\overline u_{\mu}(x)\,dxdy \nonumber\\
  &+8\int_{\R^d}\int_{\R^d}
  j_{u_{\mu}}(x)D^2_{xy}\psi(x,y)\cdot j_{u_{\mu}}(y)\,dxdy\nonumber\\
   =&4\sum_{j,k=1}^d\int_{\R^d}\int_{\R^d}
  |u_{\mu}(x)|^2\partial_{y_j}u_{\mu}(y)\partial^2_{y_jy_k}\phi(|x-y|)\partial_{y_k}\overline u_{\mu}(y)\,dxdy
  \nonumber\\
  &+
  4\sum_{j,k=1}^d\int_{\R^d}\int_{\R^d}
  |u_{\mu}(y)|^2\partial_{x_j}u_{\mu}(x)\partial^2_{x_jx_k}\phi(|x-y|)\partial_{x_k}\overline u_{\mu}(x)\,dxdy
  \nonumber
  \\
  &+8\sum_{j,k=1}^d\int_{\R^d}\int_{\R^d}
  \Im(\overline u_{\mu}(x)\partial_{x_j}u_{\mu}(x))\partial^2_{x_jy_k}\phi(|x-y|)
  \Im(\overline u_{\mu}(y)\partial_{y_k}u_{\mu}(y))\,dxdy.
  \nonumber
  \end{align}
Since $\partial_{x_j}\psi = -\partial_{y_j}\psi$, for all $j=1,\dots,n$, one can check (after a rearrangement) that the last identity of the \eqref{eq:b1a} above is equal to
   \begin{align}\label{eq:b1b}
  &
 - 4\sum_{j,k=1}^d\int_{\R^d}\int_{\R^d}\partial^2_{x_jy_k}\phi(|x-y|)|u_{\mu}(x)|^2
 \Re(\partial_{y_j}u_{\mu}(y)\partial_{y_k}\overline u_{\mu}(y))\,dxdy
  \\
  &
  -
  4\sum_{j,k=1}^d\int_{\R^d}\int_{\R^d}
  \partial^2_{x_jy_k}\phi(|x-y|)|u_{\mu}(y)|^2\Re(\partial_{x_j}u_{\mu}(x)\partial_{x_k}\overline u_{\mu}(x))\,dxdy
  \nonumber
  \\
  &
  +8\sum_{j,k=1}^d\int_{\R^d}\int_{\R^d}
  \Im(\overline u_{\mu}(x)\partial_{x_j}u_{\mu}(x))\partial^2_{x_jy_k}\phi(|x-y|)
  \Im(\overline u_{\mu}(y)\partial_{y_k}u_{\mu}(y))\,dxdy,
  \nonumber
  \end{align}
  and finally to
  \begin{align}
   &
  =-2\Big[\sum_{j,k=1}^d\int_{\R^d}\int_{\R^d} \partial^2_{x_jy_k}\phi(|x-y|) |u_{\mu}(x)|^2(\partial_{y_j}u_{\mu}(y)\partial_{y_k}\overline u_{\mu}(y)+\partial_{y_j}\overline u_{\mu}(y)\partial_{y_k}\overline u_{\mu}(y))\,dxdy
  \nonumber
  \\
   &
+\sum_{j,k=1}^d\int_{\R^d}\int_{\R^d}
  \partial^2_{x_jy_k}\phi(|x-y|)|u_{\mu}(y)|^2(\partial_{x_j}u_{\mu}(x)\partial_{x_k}\overline u_{\mu}(x)+\partial_{x_j}\overline u_{\mu}(x)\partial_{x_k}u_{\mu}(x))\,dxdy
   \nonumber
  \\
  &
  +\sum_{j,k=1}^d
  \int_{\R^d}\int_{\R^d}
   \partial^2_{x_jy_k}\phi(|x-y|) 
   (\overline u_{\mu}(x)\partial_{x_j}u_{\mu}(x)- u_{\mu}(x)\partial_{x_j}\overline u_{\mu}(x))
   \nonumber
   \\
   & \quad\quad\quad\quad\quad\quad\quad\quad\quad\quad\quad\quad
   \cdot(\overline u_{\mu}(y)\partial_{y_k}u_{\mu}(y)- u_{\mu}(y)\partial_{y_k}\overline u_{\mu}(y))\,dxdy\Big].
  \nonumber
  \end{align}
 
  If we set 
   \begin{align*}
    C^{\mu\mu}_j %= C^{\mu\mu}_j(t,x,y)
    &
    :=u_{\mu}(t,x)\partial_{y_j}\overline{u_{\mu}(t,y)}+\partial_{x_j}u_{\mu}(t,x)\overline{u_{\mu}(t,y)},
    \\
    D^{\mu\mu}_j %= D^{\mu\mu}_j(t,x,y)
    &
    :=u_{\mu}(t,x)\partial_{y_j}u_{\mu}(t,y)-\partial_{x_j}u_{\mu}(t,x)u_{\mu}(t,y),
  \end{align*}
  then by gathering \eqref{eq:b1a} and \eqref{eq:b1b} we earn   
  
   \begin{equation}\label{eq:b1f}
   B_{\mu=\kappa}
  =2\sum_{\mu=1}^N\sum_{j,k=1}^d{\textstyle \int_{\R^d}\int_{\R^d}
  \partial^2_{x_jx_k}\phi(|x-y|)\left[C^{\mu\mu}_j\overline{C^{\mu\mu}_k}+D^{\mu\mu}_j\overline{D^{\mu\mu}_k}\right]\,dxdy.}
  \end{equation}
  
  Take into account now the summand with $\mu\neq\kappa$ that is
  
  \begin{align}\label{eq:b2aI}
 B_{\mu\neq\kappa}=\sum_{\substack{\mu, \kappa=1\\ \mu\neq\kappa}}^N B^{\mu\kappa}&,
\end{align}
with the $B^{\mu\kappa}$ term given by
  
    \begin{align}\label{eq:b2a}
 B^{\mu\kappa}
 =
  4\int_{\R^d}\int_{\R^d}
  m_{u_\mu}(x)\nabla_yu_{\kappa}(y)D^2_y\psi(x,y)\nabla_y\overline u_{\kappa}(y) \,dxdy&
 \\
  + 4\int_{\R^d}\int_{\R^d}
 m_{u_{\kappa}}(y) \nabla_x u_\mu(x)D^2_x\psi(x,y)\nabla_x\overline u_\mu(x)\,dxdy &\nonumber\\
  -8\int_{\R^d}\int_{\R^d}
  j_{u_\mu}(x)D^2_{x}\psi(x,y)\cdot j_{u_\kappa}(y)\,dxdy&\nonumber\\
  =
   4\sum_{j,k=1}^d\int_{\R^d}\int_{\R^d}
  |u_\mu(x)|^2\partial_{y_j}u_{\kappa}(y)\partial^2_{y_jy_k}\phi(|x-y|)\partial_{y_k}\overline u_{\kappa}(y)\,dxdy&
 \nonumber \\
  +4\sum_{j,k=1}^d\int_{\R^d}\int_{\R^d}
  |u_{\kappa}(y)|^2\partial_{x_j}u_\mu(x)\partial^2_{x_jx_k}\phi(|x-y|)\partial_{x_k}\overline u_\mu(x)\,dxdy&
  \nonumber
   \\
  +8\sum_{j,k=1}^d\int_{\R^d}\int_{\R^d}
  \Im(\overline u_\mu(x)\partial_{x_j}u_\mu(x))\partial^2_{x_jy_k}\phi(|x-y|)
  \Im(\overline u_{\kappa}(y)\partial_{y_k}u_{\kappa}(y))\,dxdy,&
 \nonumber
  \end{align}

thus arguing as for the proof of \eqref{eq:b1f}, once one set
 \begin{align*}
    E^{\mu\kappa}_j %= E^{\mu\kappa}_j(t,x,y)
    &
    :=u_\mu(t,x)\partial_{y_j}\overline{u_\kappa(t,y)}+\partial_{x_j}u_\mu(t,x)\overline{u_\kappa(t,y)},
    \\
    F^{\mu\kappa}_j %= F^{\mu\kappa}_j(t,x,y)
    &
    :=u_\mu(t,x)\partial_{y_j}u_\kappa(t,y)-\partial_{x_j}u_\mu(t,x)u_\kappa(t,y),
  \end{align*}
we arrive at the equality
\begin{equation}\label{eq:b2f}
  B_{\mu\neq\kappa}=
  2\sum_{\substack{\mu, \kappa=1\\ \mu\neq\kappa}}^N\sum_{j,k=1}^d\int_{\R^d}\int_{\R^d}
  \partial^2_{x_jx_k}\phi(|x-y|)\left[E^{\mu\kappa}_j\overline{E^{\mu\kappa}_k}+F^{\mu\kappa}_j\overline{F^{\mu\kappa}_k}\right]\,dxdy.
   \end{equation}
  
   Therefore the identities \eqref{eq:b1f}, \eqref{eq:b2f} and the fact that $\phi$ is a convex function give $B\geq 0$. This argument implies, in combination with \eqref{eq:1}, \eqref{eq:2}, the proof of \eqref{eq:intmor2}.
  \end{proof}
  By using the identity \eqref{eq:equivalence} which appear in the proof of Lemma \ref{lem:intmor} we have an equivalent way to the \eqref{eq:intmor2} useful when the quantity $\Delta^2_x\psi(x,y)$ is nonpositive. This is contained in the following.

\begin{corollary}\label{cor:intmorbil}
 Let be $(u_\mu)_{\mu=1}^N \in\mathcal C(\R, H^1(\R^d)^N)$, $\psi=\psi(x,y)$ and $N_{(p,\psi)}$ as in Lemma \ref{lem:intmor}, then the following holds
\begin{align}
  \ddot I(t)
  &
  \geq
  -2\sum_{\mu,\kappa=1}^N\int_{\R^d}\int_{\R^d}
 m_{u_\mu}(x)m_{u_\kappa}(y)\Delta^2_x\psi(x,y)\,dxdy+N_{(p,\psi)}.
  \label{eq:intmor2a}
\end{align}
\end{corollary}

As an immediate consequence of Lemma \ref{lem:intmor} (and of Corollary \ref{cor:intmorbil}), we can now prove the following result.
\begin{proposition}\label{dim3}
  Let $d=3$, $p \in \R$ such that \eqref{eq:base} holds, 
  and let $(u_\mu)^N_{\mu=1}\in\mathcal C(\R,H^1(\R^3)^N)$ be a global solution to \eqref{eq:nls}. Then one has

\begin{align}
  &
  \sum_{\mu=1}^N\int_{\R}\int_{\R^3}|u_\mu(t,x)|^4\,dx\,dt <\infty,
  \label{eq:stima}
  \\
  &
  \sum_{\mu=1}^N \beta_{\mu\mu}\int_{\R}\int_{\R^3}\int_{\R^3}\frac{ |u_\mu(t,x)|^{2p+2}|u_\mu(t,y)|^2}{|x-y|}
  \,dx\,dy\,dt
  <\infty.  \label{eq:stima0}
\end{align}
 \end{proposition}
\begin{proof}
  Integrating \eqref{eq:intmor2a}  to time variable one obtains by \eqref{eq:intmor1}
\begin{align}
  &
 2 \sum_{\mu,\kappa=1}^N \left[\int_{\R^3}\int_{\R^3}j_{u_\mu}(t,x)\cdot\nabla_x\psi(x,y)m_{u_\kappa}(t,y)\,dxdy\right]_{t=S}^{t=T}
  \label{eq:intmor3}
  \\
  &
  \geq
  -2 \sum_{\mu,\kappa=1}^N \int_{S}^Tm_{u_\mu}(t,x)m_{u_\kappa}(t,y)\Delta^2_x\psi(x,y)\,dxdydt
  \nonumber
  \\
    &
  \ \ \ + \frac{4p}{p+1}\sum_{\mu,\kappa=1}^N\beta_{\mu\mu}\int_{S}^T\int_{\R^3}\int_{\R^3}
  |u_\mu(t,x)|^{2p+2}m_{u_\kappa}(t,y)\Delta_x\psi(x,y)\,dxdydt
  \nonumber
  \\
  &
  \ \ \ +\frac{4p}{p+1}\sum_{\substack{\mu,\nu,\kappa=1\\
  \mu\neq\nu}}^N\beta_{\mu\nu}\int_{S}^T\int_{\R^3}\int_{\R^3}
  |u_\mu(t,x)|^{p+1}|u_\nu(t,x)|^{p+1}m_{u_\kappa}(t,y)\Delta_x\psi(x,y)\,dxdydt
  \nonumber
\end{align}
Now choose $\psi(x,y)=|x-y|$. For the l.h.s of the \eqref{eq:intmor3} we have the immediate bound
\begin{align}
&
 2 \sum_{\mu,\kappa=1}^N \left[\int_{\R^3}\int_{\R^3}j_{u_\mu}(t,x)\cdot\nabla_x\psi(x,y)m_{u_\kappa}(t,y)\,dxdy\right]_{t=S}^{t=T}
\label{eq:a}
\\
&
\leq C_1\left(\sum_{\mu=1}^N\|u_\mu(T)\|_{H^1_x}
+\sum_{\mu=1}^N\|u_\mu(S)\|_{H^1_x}\right)
\nonumber
\\
&
\leq C_2\sum_{\mu=1}^N\|u_{\mu,0}\|_{H^1_x}<\infty,
\nonumber
\end{align}
for some $C_1,C_2>0$ and any $T,S\in \R$, since the $H^1_x$-norm is preserved. 
We have
\begin{equation*}
\Delta_x|x-y|=\frac{n-1}{|x-y|},
\qquad
\Delta^2_x|x-y|=-4\pi\delta_{x=y}\leq0,
\end{equation*}
and hence
\begin{align}
&
  -2 \sum_{\mu,\kappa=1}^N \int_{S}^Tm_{u_\mu}(t,x)m_{u_\kappa}(t,y)\Delta^2_x\psi(x,y)\,dxdydt
  \label{eq:b}
 \\
    &
  \ \ \ + \frac{4p}{p+1} \Bigg(\sum_{\mu,\kappa=1}^N\beta_{\mu\mu}\int_{S}^T\int_{\R^3}\int_{\R^3}
  |u_\mu(t,x)|^{2p+2}m_{u_\kappa}(t,y)\Delta_x\psi(x,y)\,dxdydt
  \nonumber
  \\
  &
  \ \ \ +\sum_{\substack{\mu,\nu,\kappa=1\\
  \mu\neq\nu}}^N\beta_{\mu\nu}\int_{S}^T\int_{\R^3}\int_{\R^3}
  |u_\mu(t,x)|^{p+1}|u_\nu(t,x)|^{p+1}m_{u_\kappa}(t,y)\Delta_x\psi(x,y)\,dxdydt\Bigg)
  \nonumber
  \\
  &
  \geq
  C\sum_{\mu=1}^N\left(\int_{S}^T \int_{\R^3}|u_\mu(t,x)|^4 \,dt\,dx+
   \beta_{\mu\mu}\int_{S}^T\int_{\R^3}\int_{\R^3}\frac{ |u_\mu(t,x)|^{2p+2}|u_\mu(t,y)|^2}{|x-y|}
  \,dx\,dy\,dt\right),
  \nonumber
\end{align}
for some $C>0$, and any $T,S\in\R$.
The thesis follows by \eqref{eq:intmor3}, \eqref{eq:a}, and \eqref{eq:b}, letting $T\to\infty, S\to -\infty$.
\end{proof}

In addition we get also 

\begin{proposition}\label{dim12}
  Let $d=1,2$, $p>0$ as in \eqref{eq:base}, and let $(u_\mu)_{\mu=1}^N\in\mathcal C(\R,H^1(\R^n)^N)$ be a global solution to \eqref{eq:nls}. 
  Then we have,
  \begin{itemize}
 \item for $d=1$
 \begin{align} \label{eq:stima1}
  \sum_{\mu=1}^N \beta_{\mu\mu}\int_{\R}\int_{\R}|u_\mu(t,x)|^{2p+4}\,dt\,dx<\infty,
  \end{align}
   \item for $d=2$
    \begin{align} \label{eq:stima2}
  \sum_{\mu=1}^N \beta_{\mu\mu}  \int_{\R}\int_{\R^2}\int_{\R^2}\frac{ |u_\mu(t,x)|^{2p+2}|u_\mu(t,y)|^2}{|x-y|}\,dt\,dx\,dy
  <\infty.
    \end{align}
\end{itemize}
 \end{proposition}

\begin{proof}

The cases $d=1,2$ can be treated by a direct application of the inequality \eqref{eq:intmor2}. Pick up once again $\psi(x,y)=|x-y|$, then we have
\begin{equation}\label{eq:delta}
\Delta_x\psi=
\begin{cases}
\frac 1{|x-y|}  \ \ \ \  \  \  \text{if} \ \ \ d=2,\\
2\delta_{x=y} \ \ \ (= D^2_x\psi)\  \,  \text{if} \ \ \ d=1.
\end{cases}
\end{equation}
Arguing as in
the proof of Proposition \ref{dim3}, we get
\begin{itemize}
 \item For $d=1$
 \begin{align} \label{eq:stimafin}
  \int_{\R}\int_{\R}|\sum_{\mu=1}^N\partial_xm_{u_\mu}(t,x)|^{2}\,dt\,dx
+ \sum_{\mu=1}^N \beta_{\mu\mu}\int_{\R}\int_{\R}|u_\mu(t,x)|^{2p+4}\,dt\,dx<\infty,
  \end{align}
  from which we infer the inequality \eqref{eq:stima1}.
   \item For $d=2$, first one needs to recall the property (fulfilled for any $d\geq1$)
\begin{align}\label{eq:stimafin2}
&\sum_{\mu,\kappa=1}^N\int_{\R}\int_{\R^d}\int_{\R^d}
  \Delta_x\psi(x,y) \nabla_xm_{u_\mu}(t,x)\cdot\nabla_y m_{u_\kappa}(t,y)\,dxdydt  \\
  &
  =\int_{\R}
   \Big\|\sum_{\mu=1}^N(-\Delta)^{\frac{1}4}m_{u_\mu}(t,x)\Big\|^2_{L^2_x}\,dt,
  \nonumber
\end{align}
(for more details see \cite{GiVel}). Then we get the following
\begin{align*}
   &\int_{\R}
   \Big\|\sum_{\mu=1}^N(-\Delta)^{\frac{1}4} |u_\mu(t,x)|^{2}\Big\|^2_{L^2_x}\,dt
 \\
   &+\sum_{\mu,\kappa=1}^N\beta_{\mu\mu}\int_{\R}\int_{\R^2}\int_{\R^2}\frac{ |u_\mu(t,x)|^{2p+2}|u_\kappa(t,y)|^2}{|x-y|}\,dt\,dx\,dy<\infty,
\end{align*}
that yields the inequality \eqref{eq:stima2}. \qedhere
\end{itemize}
\end{proof}
\begin{remark}
We observe also that, for $d=2,$ an application of the Sobolev embedding theorem implies, in similarity with the case $d=3,$ also the following bound
    \begin{align*}
  \sum^N_{\mu=1}\|u_\mu(t, x)\|^4_{L^4(\R, L_x^8)}<\infty.
  \end{align*}

\end{remark}

\begin{remark}
One could prove the interaction inequalities of the Proposition \ref{dim12} by following the theory developed for a single NLS in the
paper \cite{CGT} and based on a suitable choice of the function $\psi(x,y),$ built 
case by case. To be more precise: it is introducedfor $d=1$
\begin{equation}\label{eq:psi}
\psi(x,y)=
2\int_{-\infty}^{\frac{x-y}\varepsilon} e^{-t^2}\,dt \ \ \ \  \text{with} \ \ \varepsilon>0,
\end{equation}
and then integration by parts are performed in combination with the limiting argument $\varepsilon\rightarrow0;$
for $d=2$ it is selected a even function $\Delta^2_x\psi$ satisfying the property
\begin{align*}
-\Delta^2_x\psi=\frac {2\pi}{a}-h_a(|x-y|),
\end{align*}
for some real number $a>0$ and with
\begin{align*}
h_a(|x-y|)=
\begin{cases}
\frac 1{|x-y|^3}  \ \ \ \  \  \  \text{if} \ \ \ |x-y|\geq a,\\
\ \ 0\  \, \ \ \  \ \  \ \ \text{elesewhere},
\end{cases}
\end{align*}
then it is used a bilinear Morawetz inequality similar to \eqref{eq:intmor2}.
We elaborate our own method which is easier  to 
technicalities used above and  well-suited also to treat the case of
system with more than two nonlinear coupled equations.
\end{remark}

\section{Proof of Theorem \ref{thm:main}}\label{MainThm}

%\section{Proof of \ref{thm:main}}

We split the proof of the main Theorem \ref{thm:main} in two steps. In the first one we 
shall show, by transposing the method of \cite{Vis}, some decaying properties of 
the solution of the system \eqref{eq:nls}. In the second one we present the proof of the
scattering by combining the argument of the first step with the theory 
estabilished in \cite{Ca} and \cite{GiVel}, here applied to the case of the system of equations. 
Along this section we 
use the following further notations: given any two positive real numbers $a, b,$ we write $a\lesssim b$ 
to indicate $a\leq C b,$ with $C>0,$ we unfold the constant only when needed.
Moreover we indicate $w(t,x)=(u_\mu(t,x))_{\mu=1}^N,$ we shall use both the notations frequently
 and without much discussion. We additionally notice that Theorem 3.3.9 and Remark 3.3.12 in \cite{Ca} in connection with the defocusing nature of the system imply a well-known result concerning global well-posedness 
for \eqref{eq:nls} (see also \cite{FaMo}):
\begin{proposition}\label{ConsLaw}
  Let $1\leq d \leq 3$ and $p\in\R$ such that \eqref{eq:base} holds.
  Then for all $(u_{\mu,0})_{\mu=1}^N \in \Hin^1_x$ there exists a unique 
$(u_\mu)_{\mu=1}^N \in C(\R,\Hin^1_x) $ 
  solution to \eqref{eq:nls}, moreover 
\begin{align}
  &\norma{u_\mu(t)}_{L^2_x}=\norma{u_\mu(0)}_{L^2_x} \quad \text{ for all }\mu=1,\dots,N, \label{eq:massconservation}\\
  &E(u_1(t),\dots,u_N(t))=E(u_1(0),\dots,u_N(0))\label{eq:energyconservation},
\end{align}
with \begin{equation*}
  E(u_{1},\dots,u_{N})=\int_{\R^d} \sum_{\mu=1}^N \abs{\nabla{u_{\mu}}}^2 
  +\sum_{\mu,\nu=1 }^N
  \beta_{\mu\nu}\frac{|u_\mu u_\nu|^{p+1}}{p+1} 
         \,dx.
\end{equation*}
\end{proposition}
 \begin{remark}
   The conservation laws \eqref{eq:massconservation} and \eqref{eq:energyconservation} for the solution to \eqref{eq:nls}
   yield also that 
   \begin{equation}\label{eq:kinenergyconservation}
     \sum_{\mu=1}^N \norma{u_\mu(t)}_{ H^1_x} \leq \sum_{\mu=1}^N \norma{u_\mu(0)}_{H^1_x} < \infty.
   \end{equation}
 \end{remark}

% \begin{remark}
%   In fact something more can be proved: the system \eqref{eq:nls} is \emph{well posed}
%   in the sense of Definition 3.1.5, \cite{Ca} (see also  Theorem 4.4.6). 
% \end{remark}
 \subsection{Decay of solutions to \eqref{eq:nls}}\label{decsol}
 Our main purpose in this section is to show some decaying properties of the solution to  \eqref{eq:nls} which is a key property
 in the proof of the scattering. To be specific  one has the following.
\begin{proposition}\label{decay}
Let $1\leq d \leq 3$ and  $p\in \R$ such that \eqref{eq:base} holds. If $w\in \mathcal C(\R,\Hin^1_x),$ 
is a global solution to \eqref{eq:nls}, then we have the decay property
\begin{align}\label{eq:decay1}
\lim_{t\rightarrow \pm \infty} \|w(t)\|_{\Lin^q_x}=0,
\end{align}
with $2<q<6,$ for $d= 3$ and with $2<q<+\infty,$ for $d=1,2.$ In addition, if $d=1$
one gets
\begin{align}\label{eq:decay2}
\lim_{t\rightarrow \pm \infty} \|w(t)\|_{\Lin^\infty_x}=0.
\end{align} 
\end{proposition}

\begin{proof}
We treat only the case $t\rightarrow \infty$,
the case $t\rightarrow -\infty$ being analogous; 
we split the proof in two part: we deal first with $d=3,$ and then $d=1,2.$ \\
{\bf Case $d= 3.$} 
 Following the approach of \cite{GiVel2}
it is sufficient to prove that the property \eqref{eq:decay1} 
%\begin{equation}\label{eq:potenergy1}
%\lim_{t\rightarrow \pm \infty} \|w(t, x)\|_{\Lin^q_x}=0
%\end{equation}
for a suitable $2<q<6$, since the thesis for the general case 
can be then obtained by the conservation of mass \eqref{eq:massconservation}, the kinetic energy \eqref{eq:kinenergyconservation}  
and interpolation. In order to do this we shall prove that
\begin{equation}\label{eq:potenergy2}
\lim_{t\rightarrow \pm \infty} \|w(t)\|_{\Lin^{\frac{10}{3}}_x}=0.
\end{equation}
%In order to prove \eqref{eq:potenergy2}
For this aim we argue as in \cite{Vis}
and we assume by the absurd that there exists
$\{t_n\}$ such that
\begin{equation}\label{eq:sequencetime}
\lim_{n\rightarrow \infty} t_n=\infty \ \ \hbox{ and } \ \  \inf_n \|w(t_n, x)\|_{\Lin^{\frac{10}{3}}_x}=\epsilon_0>0.
\end{equation}
Next recall the following localized Gagliardo-Nirenberg inequality given in the Appendix \ref{appendix}
(see also \cite{L1} and \cite{L2}):
\begin{equation}\label{eq:GNloc}
\|\varphi\|_{\Lin^{\frac{2d+4}{d}}_x}^{\frac{2d+4}{d}}\leq C \left(\sup_{x\in \R^3} \|\varphi\|_{\Lin^2(Q_x)}\right)^{\frac{4}{d}} 
\|\varphi\|^2_{\Hin^1_x},
\end{equation}
where $Q_x$ is the unit cube in $\R^3$ centered in $x$.
By combining \eqref{eq:sequencetime}, \eqref{eq:GNloc} (where we choose $\varphi=w(t_n, x)$) with the bound $\|w(t_n, x)\|_{\Hin^1_x}<+\infty$,
we deduce that
\begin{equation}\label{eq:sequencespace}
\exists x_n \in \R^d \ \ \hbox{ such that } \ \ \|w(t_n, x)\|_{\Lin^2(Q_{x_n})}=\delta_0>0.
\end{equation}
We claim that
\begin{equation}\label{eq:claim0}
\exists \bar t>0 \ \ \hbox{ such that } \ \  \|w(t, x)\|_{\Lin^2(\tilde Q_{x_n})}\geq \delta_0/2, \ \ \forall t\in (t_n, t_n+\bar t),
\end{equation}
%and hence
%\begin{equation}\label{eq:claim}
%\exists \bar t>0 \hbox{ such that } \|w(t, x)\|_{\Lin^2(\tilde Q_{x_n}\times (t_n, t_n +\bar t))}=\eta_0>0
%\end{equation}
where $\tilde Q_x$ denotes the cube in $\R^d$ of radius $2$ centered in $x$.
%Notice that \eqref{eq:claim} is in contradiction with the fact that
%$w(t, x)\in L^4(\R^3\times \R)$ (that follows by the interaction Morawetz estimate).
In order to prove \eqref{eq:claim0}
we fix a cut--off function $\chi(x)\in C^\infty_0(\R^d)$ such that
$\chi(x)=1$ for $|x|<1$ and $\chi(x)=0$ for $|x|>2$.
Then by using \eqref{eq:mor1} where we choose
$\phi(x)=\chi (x-x_n)$ we get
$$\left|\frac d{dt} \int_{\R^d} \chi(x -x_n) |w(t, x)|^2 dx\right|< C \sup_t \|w(t,x)\|_{\Hin^1_x}^2.$$
Hence by \eqref{eq:kinenergyconservation} and the fundamental theorem of calculus 
we deduce

\begin{equation}
\left|\int_{\R^d} \chi(x -x_n) |w(s, x)|^2 dx - \int_{\R^d} \chi(x -x_n) |w(t, x)|^2 dx\right|\leq C_0 |t-s|,
\end{equation}
for some $C_0>0$ independent of $n$. Hence if we choose $t=t_n$ we get
the elementary inequality 

\begin{equation}
\int_{\R^d} \chi(x -x_n) |w(s, x)|^2 dx\geq  \int_{\R^d} \chi(x -x_n) |w(t_n, x)|^2 dx - C_0|t_n-s|,
\end{equation}
which implies (by the compact support property of the function $\chi$) 

\begin{equation}
\int_{\tilde Q_{x_n}} |w(s, x)|^2 dx\geq  \int_{Q_{x_n}} |w(t_n, x)|^2 dx - C_0|t_n-s|,
\end{equation}

Hence \eqref{eq:claim0} follows provided that we choose $\bar t>0$
such that
$\delta_0^2 - C_0 \bar t>\delta_0^2/4$.
The estimate \eqref{eq:claim0} contradicts the Morawetz estimates \eqref{eq:stima0}.
In fact, the lower bounds \eqref{eq:claim0} means that
\begin{equation}\label{eq:claim1}
 \sum_{\mu=1}^N \|u_\mu(t)\|^{2}_{L^{2}_x(\tilde Q_{x_n})} \geq C(d)\delta^{2}_0>0,
\end{equation}
for any $t\in (t_n, t_n+\bar t)$ with $\bar t$ as above, where we selected the intervals $t\in (t_n, t_n+\bar t)$ disjoint, 
and whence, by H\"older inequality, there exists $\bar\mu \in \{1,\dots,N\}$ such that 
\begin{equation}\label{eq:lowb1}
 \|u_{\bar\mu}(t)\|^{\bar p}_{L^{\bar p}_x(\tilde Q_{x_n})}\geq C(d)\frac {\delta^{2}_0}N, 
\end{equation}
for any $\bar p\geq 2$ and with $t\in (t_n, t_n+\bar t)$ and $\bar t$ as above. Thus we can write the following 
\begin{equation} \label{eq:stima3}
\begin{split}
  \min_{\mu=1,\dots,N} \beta_{\mu\mu}& \sum_{\mu=1}^N \int_{\R}\int_{\R^3}\int_{\R^3}\frac{ |u_\mu(t,x)|^{2p+2}|u_\mu(t,y)|^2}{|x-y|}
  \,dx\,dy\,dt \\
   &\geq  C \sum_{\mu=1}^N \sum_{n} \int_{t_n}^{t_n+\bar t}\int_{\R^d}\int_{\tilde Q_{x_n}\times\tilde Q_{x_n}} 
   |u_\mu(t,x)|^{2p+2}|u_\mu(t,y)|^2\,dx\,dy\,dt
	\\
 &\geq C   \sum_{n} \int_{t_n}^{t_n+\bar t} \delta^{4}_0\,dt=\infty,
\end{split}
\end{equation}
where in the last inequality  we used \eqref{eq:claim0} in combination with \eqref{eq:claim1}, \eqref{eq:lowb1}
and Fubini's Theorem. This leads to the contradiction with  \eqref{eq:stima0}.\\

{\bf Case $d= 1,2.$} We can argue as in the previous case just replacing the inequality \eqref{eq:GNloc}
by the following version 
\begin{equation}\label{eq:GNloc1}
\|\varphi\|_{\Lin^{3}_x}^{3}\leq C \left(\sup_{x\in \R^d} \|\varphi\|_{\Lin^2(Q_x)}\right) 
\|\varphi\|^2_{\Hin^1_x},
\end{equation}
(or alternatively by the \eqref{GNloc}) displayed in the Appendix \ref{appendix}, with the function $\varphi$ defined as above. Then  proceeding as in the previous step
we achieve, for $d=2,$ exactly the same chain of inequalities as in \eqref{eq:stima3} which is
in contradiction with \eqref{eq:stima2}. For $d=1$ we instead arrive at 
\begin{equation} \label{eq:stima4}
  \begin{split}
  \min_{\mu=1,\dots,N} &\,\beta_{\mu\mu} \sum_{\mu=1}^N  \int_{\R}\int_{\R}|u_\mu(t,x)|^{2p+4}\,dt\,dx \\
  &\geq C  \sum_{\mu=1}^N \sum_{n} \int_{t_n}^{t_n+\bar t}\int_{\tilde Q_{x_n}} |u_\mu(t,x)|^{2p+4}\,dt\,dx
    =\infty,
  \end{split}
\end{equation}
 but this contradicts the interaction estimate \eqref{eq:stima1}.
\end{proof}

\begin{remark}\label{decaylim}
As stated in the Introduction, we need to have the more stringent lower bound $\max(1, \frac 2d)$ with respect to similar one earned in Theorem 0.1 in \cite{Vis}. Indeed, if we select $0<p<1$ the coupling terms 
$\sum_{\substack{\mu,\nu=1}}^N \beta_{\mu\nu}|u_\nu|^{p+1}|u_\mu|^{p-1}u_\mu$ with $\mu\neq\nu$ give rise to a kind of nonlinearity 
which could forbid the local well-posedness result for the associated Cauchy problem \eqref{eq:nls} such as in Proposition \ref{ConsLaw}. If one replaces the nonlinear term in \eqref{eq:nls} with another model satisfying the assumptions  given in the Remark 3.3.12 in \cite{Ca}, then by repeating the argument of this section it should be possible to eliminate the lower bound conditions given in \eqref{eq:base} and \eqref{eq:al}. But as of now we are unaware of such references.
\end{remark}

 \subsection{Scattering for the NLS system \eqref{eq:nls}.}\label{NLSscat}
 This section is devoted to prove  Theorem \ref{thm:main}. 
The results are quite classic (see \cite{Ca}, \cite{GiVel2} and references therein), 
anyway we present them in the more general form of system framework. We recall from \cite{KT} the following. 

\begin{definition}\label{Sadm}
An exponent pair $(q, r)$ is Schr\"odinger-admissible if $2\leq q,r\leq \infty,$  $(q, r, n ) \neq (2,\infty, 2),$ and
\begin{align}\label{StrTV}
\frac 2q +\frac n{r}=\frac n2.
\end{align}
\end{definition} 

 In order to prove Theorem \ref{thm:main} we need the following lemma.
 \begin{lemma}\label{StriSys}
 Assume $p$ is as in \eqref{eq:base}, \eqref{eq:al}.
 Then, for any $w\in\mathcal C(\R,\Hin^1_x)$ global solution to \eqref{eq:nls}, we have
 \begin{align}
 w\in L^q(\R, \Win^{1,r}_x),
 \end{align}
for every Schr\"odinger-admissible pair $(q,r)$.
\end{lemma}
\begin{proof}
The proof is a transposition of the Theorem $7.7.3$, in \cite{Ca}. 
Let us consider the integral operator associated to \eqref{eq:nls} 
\begin{align}\label{eq:opint}
w(t+T)=e^ {it \Delta_{x}}w_0 +  \int_{0}^{t} e^ {i(t-\tau) \Delta_{x}} g(u(T+\tau),v(T+\tau),p) d\tau
      %=&e^ {it \Delta_{x}}(u_0,v_0)+  \int_{0}^{t} e^ {i(t-\tau) \Delta_{x}} (g_1(u,v,p)(T+\tau),g_2(u,v,p)(T+\tau)) d\tau,
\end{align}
where $t>T>0$ and 
\begin{gather*}
w(t)=\begin{pmatrix}
      u_1(t)\\
      \vdots\\
      u_N(t)
      \end{pmatrix}, \quad
w_{0}= \begin{pmatrix}
     u_{1,0}\\
      \vdots\\
     u_{N,0}
     \end{pmatrix}, \\
   g(w,p)=
\begin{pmatrix}  g_1(u_1,\dotsc,u_N,p)\\
\vdots
\\ g_N(u_1,\dotsc,u_N, p) \end{pmatrix} =  \begin{pmatrix}  \sum^N_{\nu=1}\beta_{1\nu} |u_1|^{p+1}|u_1|^{p-1}u_1\\
\vdots
\\ \sum^{N}_{\nu=1}\beta_{N\nu} |v_\nu|^{p+1}|u_1|^{p-1}u_1 \end{pmatrix} .
\end{gather*}
The thesis is obtained by making an use of the  classical inhomogeneous Strichartz estimates (see once again \cite{KT}).
We point out the details in handling the nonlinear part in \eqref{eq:opint},
 that is the estimate of the following
\begin{equation}\label{eq:nonlter}
\sum^N_{\mu=1}\|g_\mu(u_1,\dotsc,u_N,p)\|_{L^{q'}((T,t), W^{1,r'}_x)},
\end{equation}
for an appropriate $(q,r)$ Schr\"odinger-admissible couple:
we select $( q, r)$ such that 
\begin{equation}\label{eq:pair}
%\left(\frac 1q, \frac 1r\right)=\left( \frac {2np}{4(2p+2)} , \frac 1{2p+2}\right),
 ( q, r):= \left(\frac{4(p+1)}{np},2p+2\right).
\end{equation}

We consider for $\mu$ fixed the term 
$$g_\mu(u_1,\dotsc,u_N,p)= \sum^N_{\nu=1}\beta_{\mu\nu} |u_\nu|^{p+1}|u_\mu|^{p-1}u_\mu,$$
since the others can be handled in a similar way.
The H\"older inequality combined with Leibniz fractional rule gives
\begin{align}\label{eq.1nonl}
 \|g_\mu(u_1,\dotsc,u_N, p)\|_{L^{q'}((T,t), W^{1,r'}_x)}
 &
 \\
\leq C\big\Vert \sum^N_{\nu=1}\beta_{\mu\nu}\|u_\mu\|_{W^{1,r}_x}\|u_\nu|^{p+1}|u_\mu|^{p-1}\|_{L_x^{\frac{2p+2}{2p}}}\big\Vert_{L^{q'}((T,t))}&
\nonumber
\\
\leq C\max_{\mu,\nu=1,\dotsc,N}\beta_{\mu\nu}\big\Vert \|u_\mu\|_{W^{1,r}_x}\sum^N_{\nu=1}\|u_\nu|^{p+1}|u_\mu|^{p-1}\|_{L_x^{\frac{2p+2}{2p}}}\big\Vert_{L^{q'}((T,t))}.&
\nonumber
\end{align}
From the following pointwise Young inequality (see for instance \cite{HLP})
\begin{align*}%\label{eq.point}
|u_\nu|^{p+1}|u_\mu|^{p-1}+|u_\mu|^{p+1}|u_\nu|^{p-1}\leq C(p) \left(|u_\mu|^{2p}+|u_\nu|^{2p}\right),
 \end{align*}
 we see that the last term of the inequality above is not greater than (we set here  $\beta=\max_{\mu,\nu=1,\dotsc,N}\beta_{\mu\nu}$)
 \begin{align}\label{eq.1non2}
   &\widetilde C(p, \beta)\big\|\|u\|_{W^{1,r}_x}\sum_{\nu=1}^N\|u_\nu\|_{L_x^{2p+2}}^{2p}\big\|_{L^q((T,t))}
   \\
   &\lesssim \Big\|
   \sum_{\mu=1}^N\|u_\mu\|_{W^{1,r}_x} \cdot \sum_{\nu=1}^N\|u_\nu\|_{L_x^{2p+2}}^{2p}  
   \Big \|_{L^{q'}((T,t))}
   \nonumber
   \\
   &\lesssim\Big\|
   \sum_{\mu=1}^N\|u_\mu\|_{W^{1,r}_x}   \cdot
   \Big(\sum_{\nu=1}^N\|u_\nu\|_{L_x^{2p+2}}^{2p+1-\frac{q}{q'}}\|u_\nu\|_{L_x^{2p+2}}^{\frac{q}{q'}-1}\Big)\|_{L^{q'}((T,t))}
   \nonumber
   \\
   &\lesssim \Big\| 
   \sum_{\mu=1}^N\|u_\mu\|_{W^{1,r}_x} \cdot 
   \Big(\sum_{\nu=1}^N\|u_\nu\|_{L_x^{2p+2}}\Big)^{\frac{q}{q'}-1}
   \cdot \sum_{\kappa=1}^N\|u_\kappa\|_{L_x^{2p+2}}^{2p+1-\frac{q}{q'}}  \Big\|_{L^{q'}((T,t))},
\nonumber
\end{align}
with all the constants involved in the inequalities above independent from $t,T.$ 
Notice that $(2p+1)q'-q>0$ so the last term of the above 
chain of inequalities can be bounded by
 \begin{align}\label{eq.1non3}
   \widetilde C(p, \beta) \,
   \Big\|
   \Big(\sum_{\kappa=1}^N\|u_\kappa\|_{W^{1,r}_x}\Big)^{\frac{q}{q'}}
   \Big(\sum_{\nu=1}^N\|u_\nu\|_{L_x^{2p+2}}\Big)^{2p+1-\frac{q}{q'}}
   \Big\|_{L^{q'}((T,t))},
\end{align}
here we used without any distinction the dummy indices $\mu,\nu$ and $\kappa$ because defined on the same set. Summing in $\mu$ the \eqref{eq.1non3} above  we get that the quantity in \eqref{eq:nonlter}
is bounded by
\begin{align}\label{eq.1non3a}
  C\sup_{\tau>T}
  \Big(\sum_{\nu=1}^N\|u_\nu(\tau)\|_{L_x^{2p+2}}\Big)^{2p+1-\frac{q}{q'}}
  \Big(\sum_{\kappa=1}^N\|u_\kappa\|_{L^{q}((T,t)),W^{1,r}_x)}\Big)^{q-1},
\end{align}
with $C>0.$ The premises above, the equation \eqref{eq:opint} and the Proposition \ref{decay},
in connection with an use of the inhomogenehouse Strichartz estimates  
bring to
\begin{align}\label{eq.1non4}
\|w\|_{L^{q}((T,t), \Win^{1,r}_x)}
\leq C\|w_0\|_{\Hin^1_x}+\eta(T)\left(\|w\|_{L^{q}((T,t),\Win^{1,r}_x)}\right)^{q-1},
\end{align}
where $\eta(T)\rightarrow 0$ as $T\rightarrow \infty.$ 

Thanks to the Lemma $7.7.4$ in \cite{Ca}, 
for $T$ large enough we have
$$
\|w\|_{L^{q}((T,t),\Win^{1,r}_x)}\leq \bar C,
$$
with the constant $\bar C$ independent from $t$. In that way we get that $w\in L^{q}((T,\infty), \Win^{1,r}_x),$ and one can use a similar argument in order to have $w\in L^{q}((-\infty, -T), \Win^{1,r}_x).$
From this fact we conclude immediately that  $w\in L^q(\R, \Win^{1,r}_x)$.
\end{proof}

\begin{proof}[Proof of Theorem \ref{thm:main}]
  The proof of Theorem \ref{thm:main} is now a straightforward adaptation of Theorem 7.8.1 and Theorem 7.8.4 in \cite{Ca}:
  we shortly prove it here  for the sake of completeness.
  \\
  {\em Asymptotic completeness:} Let us write $\overline w(t)=e^ {-it \Delta_{x}}w(t)$, we get
  \begin{align}\label{eq:opint1}
\overline w(t)=w_0 +i  \int_{0}^{t} e^ {-is \Delta_{x}}  g(w,p) ds,
      %=&e^ {it \Delta_{x}}(u_0,v_0)+  \int_{0}^{t} e^ {i(t-\tau) \Delta_{x}} (g_1(u,v,p)(T+\tau),g_2(u,v,p)(T+\tau)) d\tau,
\end{align}
  moreover one has, for $0<t<t_1$,
  \begin{align}\label{eq:opint2}
\overline w(t)-\overline w(t_1)=  i\int_{t_1}^{t} e^ {-is \Delta_{x}}  g(w,p)ds.
      %=&e^ {it \Delta_{x}}(u_0,v_0)+  \int_{0}^{t} e^ {i(t-\tau) \Delta_{x}} (g_1(u,v,p)(T+\tau),g_2(u,v,p)(T+\tau)) d\tau,
\end{align}
An application of classical Strichartz estimates yields
 \begin{align}\label{eq:stric}
\|\overline w(t)-\overline w(t_1)\|_{\Hin^1_x}\lesssim &
\\
\|e^ {it\Delta_{x}} (\overline w(t)-\overline w(t_1))\|_{\Hin^1_x}\lesssim \| g(w,p)\|_{L^{q'}((t,t_1), \Win^{1,r}_x)}&
\nonumber
\end{align}
with $(q, r)$ is Schr\"odinger-admissible pair as in \eqref{eq:pair}. Following the proof of Proposition 
\ref{StriSys} we achieve 
 \begin{align}\label{eq:stric1}
\lim_{t,t_1\rightarrow \infty}\|\overline w(t)-\overline w(t_1)\|_{\Hin^1_x}=0.
\nonumber
\end{align}
Thus we can say that there exist $(u_{1,0}^{\pm},\dotsc, u_{N,0}^\pm)\in H^1(\R^d)^N$ such that
exist $(u_1(t),\dotsc, u_N(t))\rightarrow (u_{1,0}^{\pm},\dotsc, u_{N,0}^\pm)$ in $H^1(\R^d)^N$ as $t\rightarrow\pm \infty.$
Notice also that, by Proposition \ref{ConsLaw}, we have also the following properties verified
\begin{equation}\label{eq:asco1}
\|(u_{1,0}^{\pm},\dotsc, u_{N,0}^\pm)\|_{\Lin^2_x}=\|(u_{1,0},\dotsc, u_{N,0})\|_{\Lin^2_x} \ \ \ \sum_{\mu=1}^N\int_{\R^d}|\nabla u_{\mu,0}^{\pm}| dx=E(u_{1,0},\dotsc, u_{N,0})
\end{equation}

  {\em Existence of wave operators:} Let us select a Schr\"odinger-admissible pair as in \eqref{eq:pair}
  and introduce $\upsilon(t)=e^{it\Delta_x}w^{+}_0$ (the proof for
   $w^{-}_0(t)$ is analogous). Then by the Strichartz estimates and 
  Corollary 2.3.7 in \cite{Ca} we get that, for $T>0$,
 \begin{equation} \label{eq:set0}
 \widetilde{\mathscr{H}}(T)=\|\upsilon(t)\|_{L^q([T, \infty], \Win^{1,r}_x)}+\sup_{t\geq T}\|\upsilon(t)\|_{\Lin^r_x},\end{equation}
 is a decreasing function w.r.t. the $T$ variable and such that $ \widetilde{\mathscr{H}}(T)\rightarrow 0$ as $T\rightarrow \infty.$
As a consequence we are allowed to introduce the complete metric space $Z\subseteq L^q([T, \infty), \Win^{1,r}_x)$ defined as
 \begin{align}\label{eq:set1}
Z=\big\{ w(t) : \|w(t)\|_{L^q((T, \infty), \Win^{1,r}_x)}+\sup_{t\geq T}\|w(t)\|_{\Lin^r_x}\leq 2\widetilde{\mathscr{H}}(T) \big\}
\end{align}
and equipped with the topology induced by $\|.\|_{L^q([T, \infty), \Lin^r_x)}.$
Let be
\begin{align}\label{eq:scat1}
\mathcal I (w)(t)=-i\int_{t}^{\pm\infty}e^{i(t-\tau)\Delta_x}g(u(\tau),v(\tau),p)d\tau,
\end{align}
with
$$
\mathcal I (w)\in C([T,\infty), \Hin^1_x)\cap L^q((T, \infty), \Win^{1,r}_x).
$$
 (see \cite{KT} or the Corollary 2.3.6 in \cite{Ca}, easily generalizable to a system of equations).
Furthermore, by the inequality (see Lemma \ref{StriSys})
$$
\| g(w,p)\|_{L^{q'}((T,\infty), \Win^{1,r}_x)}\leq C (\widetilde{\mathscr{H}}(T))^{2p+1}
$$
with $w=(u,v)\in Z,$ in combination with the behavior of $\widetilde{\mathscr{H}}(T)$ for $T$ large enough and the Sobolev embedding inequality we achieve the following estimate
\begin{equation} \label{eq:set2}
 \|w(t)\|_{L^q((T, \infty), \Win^{1,r}_x)}+\sup_{t\geq T}\|w(t)\|_{L^\infty((T, \infty), \Lin^r_x)}\leq\widetilde{\mathscr{H}}(T),
 \end{equation}
for $T$ large enough. By the estimate \eqref{eq:set0} and the \eqref{eq:set2} above we conclude that
the operator 
\begin{align}\label{eq:fixpoint}
\mathcal K(w)=e^{it\Delta_x}w_0^{+}+\mathcal I (w)
\end{align}
is a contraction on $Z$ with respect the norm $\|.\|_{L^q([T, \infty), \Win^{1,r}_x)}.$ By applying a fixed
point argument we get that there exists $w\in Z$ satisfying the equation \eqref{eq:fixpoint}. In addition
$w\in C([T,\infty), \Hin^1_x).$ By classical arguments one can show also that $w$ is a global solution to the equation \eqref{eq:nls} and then $w(0)=w_0\in \Hin^1_x$ is well defined.
Furthermore the properties \eqref{eq:scattering} is fulfilled. The proof of the remaining part regarding the uniqueness reads as in Theorem 7.8.4 in \cite{Ca}, 
so we skip it.
\end{proof}

\appendix
\section{A localized Gagliardo-Nirenberg inequality}
\label{appendix}
The principal target of this section is to prove of the localized inequality \eqref{eq:GNloc} used in the proof of
Proposition \ref{decay} (see Section \ref{decsol}). Albeit it already appeared
in the literature (see for example \cite{Vis}, \cite{L1,L2} or \cite{TeTzVi} in the context of product space $\R^d\times M,$ with $M^k$ any $k$-dimensional compact manifold), we recall it in a more general form. We have
\begin{proposition}\label{GNloc}
Let be $d\geq 1$ and $\alpha\in \N$, then for all vector-valued functions $\phi=(\phi_\ell)_{\ell=1}^\alpha \in H^1(\R^d)^\alpha$ one gets the following 
\begin{equation}\label{eq:GNlocN}
\|\phi\|_{L^{\frac{2d+4}{d}}(\R^d)^\alpha}^{\frac{2d+4}{d}}\leq C \left(\sup_{x\in \R^d} \|\phi\|_{ L^2(Q_x)^\alpha}\right)^{\frac{4}{d}} 
\|\phi\|^2_{ H^1(\R^d)^\alpha}.
\end{equation}
\end{proposition}
\begin{proof}
Consider an open covering of $\R^d$ given by a family of disjoint cubes $\{Q_s\}_{s\in\N}$.

Let us look at the high dimensional case $d\geq 3.$ For any component of $\phi=(\phi_1,\dots, \phi_\alpha)$ one has that $\phi_\ell\in L^{\frac{2d}{d-2}}, \, \ell=1,\dots,\alpha$, then the Sobolev embedding and 
an application of H\"older inequality bring to the chain of inequalities

\begin{align}\label{eq:GNlocN2}
\sum_{\ell=1}^\alpha\int_{Q_s}|\phi_\ell|^{\frac{2d+4}{d}}\leq C \sum_{\ell=1}^\alpha\left(\int_{Q_s}|\phi_\ell|^{2}\right)^{
\frac 2d} 
\left(\int_{Q_s}|\phi_\ell|^{\frac{2d}{d-2}}\right)^{\frac{d-2}{d}}&
\\
\leq C \sum_{\ell=1}^\alpha\left(\int_{Q_s}|\phi_\ell|^{2}\right)^{
\frac 2d} 
\left(\int_{Q_s}|\phi_\ell|^{\frac{2d}{d-2}}\right)^{\frac{2(d-2)}{2d}}\leq C(\alpha) \left(\sum_{\ell=1}^\alpha\|\phi_\ell\|_{ L^2(Q_s)}\right)^{
\frac 4d} 
\sum_{\ell=1}^\alpha\|\phi_\ell\|^2_{H^1(Q_s)}.&
\nonumber
\end{align}
The estimates above can be rewritten as 
\begin{equation}\label{eq:GNlocN3}
\|\phi\|_{L^{\frac{2d+4}{d}}(Q_s)^\alpha}^{\frac{2d+4}{d}}\leq C \left(\|\phi\|_{L^2(Q_s)^\alpha}\right)^{\frac{4}{d}} 
\|\phi\|^2_{H^1(Q_s)^\alpha},
\end{equation}
and hence summing over $s$ we arrive at
 \begin{align}\label{eq:GNlocN4}
\|\phi\|_{L^{\frac{2d+4}{d}}(\R^d)^\alpha}^{\frac{2d+4}{d}}\leq C \left(\sup_{s\in\N}\|\phi\|_{L^2(Q_s)^\alpha}\right)^{\frac{4}{d}} 
\sum_{s\in\N}\|\phi\|^2_{H^1(Q_s)^\alpha}&
\nonumber\\
\leq C \left(\sup_{s\in\N}\|\phi\|_{L^2(Q_s)^\alpha}\right)^{\frac{4}{d}} 
\|\phi\|^2_{H^1(\R^d)^\alpha},&
\end{align}
note the fact that the estimate above is translation invariant so the constants are independent from $s.$ Therefore the estimate \eqref{eq:GNlocN} follows from the above \eqref{eq:GNlocN4} in combination with the fact $\{Q_s\}_{s\in\N}\subset\{Q_x\}_{x\in\R^d}.$\\

The remaining cases, $d=1,2,$ can be handled in the same way as before with minor changes.
For $d=2$ we need to replace the estimate \eqref{eq:GNlocN2} by the following
\begin{align}\label{eq:GNlocN5}
\sum_{\ell=1}^\alpha\int_{Q_s}|\phi_\ell|^{4}
\\
\leq C \sum_{\ell=1}^\alpha\left(\int_{Q_s}|\phi_\ell|^{2}\right) 
\|\phi_\ell\|^2_{H^1(Q_s)}\leq C(\alpha) \left(\sum_{\ell=1}^\alpha\|\phi_\ell\|_{ L^2(Q_s)}\right)^{2} 
\sum_{\ell=1}^\alpha\|\phi_\ell\|^2_{H^1(Q_s)},&
\nonumber
\end{align}
which can be carried out taking $u=\abs{\phi_l}^2$ in the following Sobolev inequality
\begin{equation*}
  \norma{u}_{L^2(Q_s)}\lesssim \norma{u}_{L^1(Q_s)} + \norma{\nabla u}_{L^1(Q_s)},
\end{equation*}
 and by an use of Leibniz chain rule.
 Then one argues as in the proof for the higher dimensions case. For the last case, that is $d=1,$ we use instead of \eqref{eq:GNlocN5}
the following
\begin{align}\label{eq:GNlocN5bis}
\sum_{\ell=1}^\alpha\int_{Q_s}|\phi_\ell|^{6}
\\
\leq C \sum_{\ell=1}^\alpha\left(\int_{Q_s}|\phi_\ell|^{2}\right)^2 
\|\phi_\ell\|^2_{H^1(Q_s)}\leq C \left(\sum_{\ell=1}^\alpha\|\phi_\ell\|_{ L^2(Q_s)}\right)^{4} 
\sum_{\ell=1}^\alpha\|\phi_\ell\|^2_{H^1(Q_s)},&
\nonumber
\end{align}
which can be earned by using once again inhomogeneous Sobolev embedding (in details, $W^{1,1}_x \subset L^\infty_x$) and Leibniz chain rule. 
Then one continues as above.
\end{proof}
\begin{remark}\label{GNlocN3}
Following the paper \cite{Vis}, we can also obtain a variant of the inequality \eqref{eq:GNlocN} in the cases $d=1,2.$
The Sobolev embedding $W^{1,1}_x \subset L^2_x$ enables us to write the easy localized estimate
\begin{align}\label{eq:GNlocN5th}
\sum_{\ell=1}^\alpha\int_{Q_s}|\phi_\ell|^{3}
\\
\leq C \sum_{\ell=1}^\alpha\left(\int_{Q_s}|\phi_\ell|^{2}\right)^{\frac 12}
\||\phi_\ell|^2\|_{W^{1,1}(Q_s)}\leq C(\alpha) \left(\sum_{\ell=1}^\alpha\|\phi_\ell\|_{ L^2(Q_s)}\right) 
\sum_{\ell=1}^\alpha\|\phi_\ell\|_{H^1(Q_s)}^2,&
\nonumber
\end{align}
this fact, arguing as in the proof of the Lemma \ref{GNloc} brings to the estimate 
\begin{equation}\label{eq:GNloc12}
\|\varphi\|_{ L^{3}(\R^d)^\alpha}^{3}\leq C \left(\sup_{x\in \R^d} \|\phi\|_{ L^2(Q_x)^\alpha}\right)
\|\phi\|^2_{ H^1(\R^d)^\alpha},
\end{equation}
that is the inequality \eqref{eq:GNloc1}.
\end{remark}

{\it {\bf Acknowledgements:}}
The authors are grateful to Luca Fanelli and Nicola Visciglia for pointing this recent developments and for interesting and helpful discussions.

\end{document}